\documentclass[10pt, reqno]{amsart}

\usepackage{geometry}
\usepackage[all]{xy}
\usepackage[toc,page]{appendix}
\usepackage{epstopdf}
\usepackage{amssymb,amsmath,amsfonts,amsthm,graphicx,enumerate,amscd}
\usepackage{mathrsfs}
\usepackage{youngtab}
\usepackage[bottom]{footmisc}

\theoremstyle{plain}
\newtheorem{theorem}{Theorem}[section]

\newtheorem{proposition}[theorem]{Proposition}

\theoremstyle{definition}

\newtheorem{remark}[theorem]{Remark}

\newtheorem{definition}[theorem]{Definition}

\allowdisplaybreaks

\title{Adams operations on classical compact Lie groups}
\author{Chi-Kwong Fok}
\date{July 21, 2016}

\begin{document}
\maketitle

\begin{abstract}
	Let $G$ be $U(n)$, $SU(n)$, $Sp(n)$ or $Spin(n)$. In this short note we give explicit general formulas for Adams operations on $K^*(G)$, and eigenvectors of Adams operations on $K^*(U(n))$.
\end{abstract}

\maketitle


%\bibliographystyle{amsplain}
\section{Introduction}
Adams operations are important cohomological operations on $K$-theory. To see their importance one needs to look no further than their application to the famous problem of finding parallelizable spheres (cf. \cite{A}). Another application of Adams operations is the extraction of certain information on homotopy groups of $H$-spaces and Lie groups in particular (cf. \cite{Bou}, \cite{D}, \cite{DP}). With this in mind, we present the computation of Adams operations on compact classical Lie groups in this note. 

The determination of Adams operations on compact Lie groups was first carried out by Naylor (cf. \cite{N}). There he gave an explicit formula for $\psi^2$ on $K^*(SU(n))$ and, by means of repeated applications of this formula and Lagrange interpolation, wrote down a description of $\psi^l$, which is not as explicit. He also provided a recursive algorithm to find the eigenvectors of $\psi^l$ without giving an explicit formula. For the other compact classical Lie groups, he suggested ways of computing Adams operations based on the result for $SU(n)$ and the functoriality of $\psi^l$, and no explicit formulas were presented.

Adams operations on the rank 2 compact Lie groups $SU(3)$, $Sp(2)$ and $G_2$ were computed explicitly in \cite{W2} and \cite{W3} by means of the Chern character isomorphism on those groups (cf. \cite{W}). It would be difficult to extend this technique to yield formulas for other higher rank classical Lie groups due to overwhelming computational complexity. Adams operations $\psi^l$ for $l=-1, 2, 3$ on $SU(n)$ with respect to judiciously chosen generators of $K^*(SU(n))$ were given in \cite[Proposition 3.3]{DP}.  Ever since then no further formulas for Adams operations on other classical Lie groups have appeared. However, Adams operations on the exceptional Lie groups and their eigenvectors were completely settled in \cite{D} by Davis, who also determined there the $v_1$-periodic homotopy groups of exceptional Lie groups. A localization of the actual homotopy groups, $v_1$-periodic homotopy groups indicate roughly the portion which is detected by $K$-theory and its operations (cf. \cite[Section 1]{D}). For a compact Lie group its $v_1$-periodic homotopy group is a direct summand of some actual homotopy group. Davis accomplished the computation of the $v_1$-periodic homotopy groups of exceptional Lie groups using a theorem of Bousfield (cf. \cite[Theorem 9.2]{Bou}), which roughly asserts that the $v_1$-periodic homotopy groups can be computed using Adams operations. 

The aim of this note is to complete the computation of Adams operations on compact Lie groups by building on the previous work mentioned above and give explicit general formulas for compact classical Lie groups $G=U(n)$, $SU(n)$, $Sp(n)$, $Spin(n)$, and the associated eigenvectors for $U(n)$. We take advantage of the recent results of \cite{BZ} on equivariant $K$-theory of $G$ to get the equivariant version of Adams operations on $U(n)$. The formula for the non-equivariant unitary case then drops out on applying the forgetful map. At the same time we recover the formula for $\psi^2$ in \cite[Theorem 1.2]{N}, and find that our formula does agree with the formulas, given in \cite[pp. 149-150]{N}, for $\psi^l$ on $SU(n)$ when $n$ is $3, 4$ and $5$. The symplectic and spin cases can be settled using the result in the unitary case, the functoriality of Adams operations as suggested in \cite{N}, and some representation theory. We also recover, using our results, the formulas for Adams operations on $G_2$ in \cite{W3} without appealing to the Chern character isomorphism. The formula for eigenvectors of Adams operations on $U(n)$ presented in this note is explicit in that it does not require the knowledge of other eigenvectors, while the recursive algorithm given in \cite{N} does. 

Before listing the formulas for Adams operations on $G$ in Theorems \ref{mainthm}, \ref{symplectic} and \ref{spin}, we would like to first explain and define a few notations. The map $\delta$, which maps $R(G)$ to $K^{-1}(G)$, is defined in Definition 2.1. The main theorem in \cite{Ho} (cf. Theorem \ref{K_G} (\ref{Ho})) states that the classes in $\delta(R(G))\subset K^{-1}(G)$ generate the $K$-theory ring of $G$. Because of this fact the following Theorems give a complete description of Adams operations on the ring $K^*(G)$. We use $\sigma_n$ to denote the standard representation of $n$ dimensions. 

\begin{definition}
	\begin{enumerate}
		\item For nonnegative integers $k$ and $p$ and positive integers $n$ and $l$, we define $\mu(n, l, k, p)$ to be the cardinality of the set 
	\[\{(k_1, \cdots, k_n)\in\mathbb{Z}^{\oplus n}|\ 0\leq k_r\leq l-1\text{ for }1\leq r\leq n,\ k_1+\cdots+k_n=lk-p\}.\]
			We also define
			\begin{align*}\alpha(n, l, k, p)&=\mu(n, l, k, p)+\mu(n, l, k, n-p), \text{and}\\ \beta(n, l, k, p)&=\mu(n, l, k, p)-\mu(n, l, k, n-p).\end{align*}
		\item Let $\{p_j(y)\}_{j=0}^\infty$ be the sequence of polynomials which are nonzero coefficients of the Taylor series of $\displaystyle \left(\frac{t}{\sinh t}\right)^y$, i.e.  
			\[\left(\frac{t}{\sinh t}\right)^y=\sum_{j=0}^\infty p_j(y)t^{2j}.\]
	\end{enumerate}
\end{definition}

%we remark that the polynomial sequence appearing in the formula involves such combinatorial objects as Bernoulli numbers, and seems to bear relations to $\widehat{A}$-genus.

	\begin{theorem}[(Special) unitary case]\label{mainthm} 
		\begin{enumerate}
			\item\label{unitary} For $G=U(n)$, we have
		\begin{eqnarray}\label{unitaryadamsop}\psi^l(\delta(\bigwedge\nolimits^k\sigma_n))=(-1)^k l\sum_{p=1}^{n}(-1)^{p}\mu(n, l, k, p)\delta(\bigwedge\nolimits^p\sigma_n). \end{eqnarray}
		%where $\displaystyle\mu(n, l, k, p)=\sum_{q=0}^{k-1}(-1)^q\binom{n}{q}\binom{n+l(k-q)-p-1}{n-1}$, the cardinality of the set
		%\[\{(k_1, \cdots, k_n)\in\mathbb{Z}^{\oplus n}|\ 0\leq k_r\leq l-1\text{ for }1\leq r\leq n,\ k_1+\cdots+k_n=lk-p\}.\]
		The number $\mu(n, l, k, p)$ is equal to 
		\begin{eqnarray}\label{adamsopcoeff}\sum_{q=0}^{k-1}(-1)^q\binom{n}{q}\binom{n+l(k-q)-p-1}{n-1}.\end{eqnarray}
		In particular, when $l=2$, 
		\[\psi^2(\delta(\bigwedge\nolimits^k\sigma_n))=(-1)^k\cdot 2\sum_{p=1}^{2k}(-1)^p\binom{n}{2k-p}\delta(\bigwedge\nolimits^p\sigma_n).\]
		The formula for $G=SU(n)$ is the same except that $\delta(\bigwedge\nolimits^n\sigma_n)$ becomes 0 in this case.
			\item\label{eigenvector}  An eigenvector of $\psi^l\otimes\text{Id}_\mathbb{Q}: K^*(U(n))\otimes\mathbb{Q}\to K^*(U(n))\otimes\mathbb{Q}$ (for $l$ any integer) corresponding to the eigenvalue $l^{n-k}$, for $k=0, 1, \cdots, n-1$, is
		\begin{eqnarray}\label{eigvec}\sum_{i=1}^n (-1)^{i-1}\left(\sum_{j=0}^{\lfloor\frac{k}{2}\rfloor}\frac{p_j(n)}{(k-2j)!}(n-2i)^{k-2j}\right)\delta(\bigwedge\nolimits^i\sigma_n).\end{eqnarray}
		Moreover, the polynomial $p_j(y)$ is of degree $j$ and satisfies the recurrence relation
		\[p_0(y)=1, p_j(y)=-\frac{y}{2j}\sum_{k=1}^j\frac{2^{2k}B_{2k}}{(2k)!}p_{j-k}(y),\]
		where $B_{2k}$ is the $2k$-th Bernoulli number.
		\end{enumerate}
	\end{theorem}
		%These polynomials are defined by the equations
		 %\[p_j(2j+2k+1)=\left.\frac{d^{2k}}{dx^{2k}}\right|_{x=\frac{2j+2k+1}{2}}\frac{2^{2j}}{(2(j+k))!}(x-1)(x-2)\cdots (x-2(j+k))\ \ \text{for }k=0, 1, \cdots, j.\]
		 %Alternatively, 
		%The first 6 members of the polynomial sequence are
		%\begin{align*}
		%	p_0(y)&=1\\
		%	p_1(y)&=-\frac{y}{6}\\
		%	p_2(y)&=\frac{y(5y+2)}{360}\\
		%	p_3(y)&=-\frac{y(35y^2+42y+16)}{45360}\\
		%	p_4(y)&=\frac{y(5y+4)(35y^2+56y+36)}{5443200}\\
		%	p_5(y)&=-\frac{y(385y^4+1540y^3+2684y^2+2288y+768)}{359251200}
		%\end{align*}	
	\begin{theorem}[Symplectic case]\label{symplectic} For $G=Sp(n)$, $l>0$ and $1\leq k\leq n$, we have
	\begin{align*}
		\psi^l(\delta(\bigwedge\nolimits^k\sigma_{2n}))&=(-1)^kl \left(\sum_{p=1}^{n-1}(-1)^p\alpha(2n, l, k, p)\delta(\bigwedge\nolimits^p\sigma_{2n})\right.\\
												&+(-1)^n\mu(2n, l, k, n)\delta(\bigwedge\nolimits^n\sigma_{2n})\Bigg).
	\end{align*}
	\end{theorem}
	\begin{theorem}[Spin case]\label{spin}
		\begin{enumerate}
			\item\label{spinodd} For $G=Spin(2n+1)$, $l>0$ and $1\leq k\leq n-1$, we have
				\begin{align*}
										&\psi^l(\delta(\bigwedge\nolimits^k\sigma_{2n+1}))\\
										  =&(-1)^k l\left(\sum_{p=1}^{n-1}((-1)^p\beta(2n+1, l, k, p)-(-1)^n\beta(2n+1, l, k, n))\delta(\bigwedge\nolimits^p\sigma_{2n+1})\right.\\
										  &+(-1)^n\cdot 2^{n+1}\beta(2n+1, l, k, n)\delta(S)\Bigg),\text{ and}\\
	                                                     			&\psi^l(\delta(S))\\
	                                                     			=&\frac{l}{2^{n+1}}\sum_{p=1}^{n-1}\sum_{k=1}^n(-1)^k((-1)^{p}\beta(2n+1, l, k, p)-(-1)^{n}\beta(2n+1, l, k, n))\delta(\bigwedge\nolimits^p\sigma_{2n+1})\\
				 						 &+l\sum_{k=1}^n(-1)^{k+n}\beta(2n+1, l, k, n)\delta(S), 
				\end{align*}	
				where $S$ is the spin representation. 
			\item\label{spineven} For $G=Spin(2n)$ and $1\leq k\leq n-2$, we have
			\begin{align*}
													&\psi^l(\delta(\bigwedge\nolimits^k\sigma_{2n}))\\
												      =&(-1)^{k+n}l\left(\sum_{p\geq 1}((\alpha(2n, l, k, n-2p)\right.-2\mu(2n, l, k, n))\delta(\bigwedge\nolimits^{n-2p}\sigma_{2n})\\
												      &-(\alpha(2n, l, k, n-2p-1)-\alpha(2n, l, k, n-1))\delta(\bigwedge\nolimits^{n-1-2p}\sigma_{2n}))\\
												      &-2^{n-1}(\alpha(2n, l, k, n-1)-2\mu(2n, l, k, n))(\delta(S_+)+\delta(S_-))\Bigg), \\
						&\psi^l(\delta(S_\pm))\\
					       =&\left(\frac{1}{2}\right)^{n}l\sum_{k\geq 0}\sum_{p\geq 1}((\alpha(2n, l, n-2k-1, n-2p-1)\\
					       &-\alpha(2n, l, n-2k-1, n-1))\delta(\bigwedge\nolimits^{n-2p-1}\sigma_{2n})\\
					       &-(\alpha(2n, l, n-2k-1, n-2p)-2\mu(2n, l, n-2k-1, n))\delta(\bigwedge\nolimits^{n-2p}\sigma_{2n}))\\
					       &+\frac{l}{2}\sum_{k\geq 0}((\alpha(2n, l, n-2k-1, n-1)-2\mu(2n, l, n-2k-1, n)\pm l^{n-1})\delta(S_+)\\
					       &+(\alpha(2n, l, n-2k-1, n-1)-2\mu(2n, l, n-2k-1, n)\mp l^{n-1})\delta(S_-)),
		                  \end{align*}	
where $S_+$ and $S_-$ are positive and negative half spin representations respectively. 
		\end{enumerate}
	\end{theorem}
	In light of the aforementioned results in \cite{Bou} and \cite{D}, we hope that our formulas will be of interest for future research on the alternative way of using Adams operations to find $v_1$-periodic homotopy groups of classical compact Lie groups (earlier results on those homotopy groups, obtained by means of homotopy-theoretic and unstable Novikov spectral sequence methods, can be found in the references therein). 
	
	\textbf{Acknowledgments}\ \ We would like to thank Donald Davis for answering questions and giving comments on the first draft of this note. We are also grateful to Nan-Kuo Ho, Michael Mandell, Reyer Sjamaar and the anonymous referee for suggestions for improving the exposition of this note.
	
\section{$G=U(n)$ or $SU(n)$}

Throughout this paper, $K_G^*(X)$ denotes the $\mathbb{Z}/2$-graded equivariant complex $K$-theory of a compact Hausdorff $G$-space $X$. We shall first give a quick review of the results from \cite{Ho} and \cite{BZ} on the (equivariant) $K$-theory of a compact Lie group $G$ with conjugation action by itself. 

\begin{definition}[\cite{BZ}\footnote{The definition of $\delta_G$ first appeared in \cite{BZ} but is incorrect. Here we use a corrected version as in \cite[Proposition 2.2 and Definition 2.5]{F}.}] Let $\delta_G: R(G)\to K^{-1}_G(G)$ be the map which sends a representation $\rho$ to the following complex of equivariant $G$-vector bundles
\begin{align*}
	0\longrightarrow G\times\mathbb{R}\times V&\longrightarrow G\times\mathbb{R}\times V\longrightarrow 0, \\
	(g, t, v)&\mapsto (g, t, -t\rho(g)v)\ \ \text{if }t\geq 0,\ \textnormal{and} \\
	(g, t, v)&\mapsto (g, t, tv)\ \ \text{if }t\leq 0, 
\end{align*}
where $V$ is the underlying complex vector space of $\rho$. Similarly, we also define the non-equivariant version $\delta: R(G)\to K^{-1}(G)$ which is $\delta_G$ composed with the forgetful map $K_G^*(G)\to K^*(G)$.
\end{definition}
The map $\delta_G$ (resp. $\delta$) is a derivation of $R(G)$ taking values in $K_G^{-1}(G)$ as an $R(G)$-module (resp. taking values in $K^{-1}(G)$ as an $R(G)$-module whose module structure is given by the augmentation map). See \cite[Equation (2)]{Ho} and \cite{BZ}. 
\begin{theorem}\label{K_G}
	If $G$ is a compact connected Lie group with torsion-free fundamental group, $T$ a maximal torus and $W$ the Weyl group, then we have the following.
	\begin{enumerate}
		\item\label{Ho}\cite[Theorem A(iii) and (v)]{Ho} There is a ring isomorphism
		\[K^*(G)\cong\bigwedge_\mathbb{Z}\nolimits^*\text{Im}(\delta).\]
		In particular, if $G$ is simply-connected of rank $r$, then
		\[K^*(G)\cong \bigwedge_\mathbb{Z}\nolimits^*(\delta(\rho_1), \cdots, \delta(\rho_r)),\]
		where $\rho_1, \cdots, \rho_r$ are fundamental representations of $G$.
		\item\label{inj}\cite{BZ} The restriction map 
		\[p_G^*: K_G^*(G)\to K_T^*(T)\]
		is injective and $\text{Im}(p_G^*)=K_T^*(T)^W$. 
		\item\cite{BZ} Let $\Omega^*_{R(G)/\mathbb{Z}}$ be the ring of K\"ahler differentials of $R(G)$ over $\mathbb{Z}$, and $\varphi: \Omega_{R(G)/\mathbb{Z}}^*\to K_G^*(G)$ be the $R(G)$-algebra homomorphism defined by 
			\begin{enumerate}
				\item $\varphi(\rho_V):=[G\times V]\in K_G^0(G)$, and
				\item $\varphi(d\rho_V)=\delta_G(\rho_V)$.
			\end{enumerate}
			Then $\varphi$ is an isomorphism.
	\end{enumerate}
\end{theorem} 

\begin{remark}\label{primindecomp}
	Let $I$ be the augmentation ideal of $R(G)$. So $I/I^2$ is the free $\mathbb{Z}$-module of `indecomposables'. Theorem \ref{K_G} (\ref{Ho}) and the derivation property of $\delta$ imply that $\delta$ induces an isomorphism between $I/I^2$ and the free $\mathbb{Z}$-module generated by the primitive elements of $K^*(G)$. 
\end{remark}
Now let us consider Adams operations in the case $G=U(n)$. Let $\sigma_n$ be its standard representation. Recall that $R(U(n))\cong\mathbb{Z}[\sigma_n, \bigwedge\nolimits^2\sigma_n, \cdots, \bigwedge\nolimits^n\sigma_n, \bigwedge\nolimits^n\overline{\sigma_n}]$. By Theorem \ref{K_G}, $\delta_G(\sigma_n), \cdots, \delta_G(\bigwedge\nolimits^n\sigma_n)$ (resp. $\delta(\sigma_n), \cdots, \delta(\bigwedge\nolimits^n\sigma_n)$) are primitive generators of $K_{U(n)}^*(U(n))$ (resp, $K^*(U(n))$). Our strategy for computing Adams operations on $K^*(U(n))$ is to first compute those on $K_{U(n)}^*(U(n))$ or, equivalently, on the generators $\delta_G(\sigma_n), \cdots, \delta_G(\bigwedge\nolimits^n\sigma_n)$, and then apply the forgetful map. We choose to compute Adams operations via the equivariant calculation in view of Theorem \ref{K_G} (\ref{inj}), which enables us to pass the computation to $K_T^*(T)$, as well as the following `splitting principle' which helps simplify the computations.

\begin{proposition}
	Let $\lambda_1, \cdots, \lambda_n$ be the standard basis vectors of the weight lattice of $T$. Viewing $K_T^*(T)$ as $R(T)\otimes K^*(T)$, we have
	\[p_G^*\delta_G(\bigwedge\nolimits^k\sigma_n)=\sum_{1\leq i_1<\cdots<i_k\leq n}\lambda_{i_1}\cdots\lambda_{i_k}\otimes(t_{i_1}+\cdots+t_{i_k}),\]
	where $t_i=\delta(\lambda_i)$. 
\end{proposition}
\begin{proof}
	The complex of equivariant $G$-vector bundles representing $\delta_G(\bigwedge\nolimits^k\sigma_n)$ is decomposed, on restriction to $T$, into a direct sum of complexes of 1-dimensional equivariant $T$-vector bundles, each of which corresponds to a weight of $\bigwedge\nolimits^k\sigma_n$. The Proposition then follows immediately. 
\end{proof}

By definition of Adams operations, $\psi^l(\lambda_i)=\lambda_i^l$ and $\psi^l(t_i)=lt_i$. The next task is to rewrite 
\[\psi^l(p_G^*\delta_G(\bigwedge\nolimits^k\sigma_n))=l\sum_{1\leq i_1<\cdots<i_k\leq n}\lambda_{i_1}^l\cdots\lambda_{i_k}^l\otimes(t_{i_1}+\cdots+t_{i_k})\]
as a linear combination of $p_G^*\delta_G(\sigma_n), \cdots, p_G^*\delta_G(\bigwedge\nolimits^n\sigma_n)$ with coefficients in $R(U(n))\cong R(T)^{S_n}$. 

\begin{proposition}\label{sympoly}
		We have, in $R(T)\otimes K^*(T)$, 
		\begin{eqnarray}
			p_G^*\delta_G(\bigwedge\nolimits^k\sigma_n)=\sum_{i=1}^n\sum_{j=1}^k(-1)^{j+1}s_{k-j}\lambda_i^j\otimes t_i,\text{ and}\label{esym}\\
			\sum_{i=1}^n\lambda_i^k\otimes t_i=\sum_{j=1}^k(-1)^{j+1}h_{k-j}p_G^*\delta(\bigwedge\nolimits^j\sigma_n), \label{hsym}
		\end{eqnarray}
		where $s_k$ is the $k$-th elementary symmetric polynomial in $\lambda_1, \cdots, \lambda_n$, and $h_k$ is the $k$-th complete homogeneous symmetric polynomial in $\lambda_1, \cdots, \lambda_n$.
\end{proposition}

\begin{proof} Equation (\ref{esym}) can be obtained as follows:
		\begin{align*} 
				p_G^*(\delta_G(\bigwedge\nolimits^k\sigma_n))&=\sum_{1\leq i_1<\cdots<i_k\leq n}\lambda_{i_1}\cdots\lambda_{i_k}\otimes(t_{i_1}+\cdots+t_{i_k})\\
				                                                                 &=\sum_{i=1}^n\lambda_i\sum_{\substack{1\leq i_1<\cdots< i_{k-1}\leq n\\i\neq i_j\text{ for all }1\leq j\leq k-1}}\lambda_{i_1}\cdots\lambda_{i_{k-1}}\otimes t_i\\
				                                                                 &=\sum_{i=1}^n\lambda_i\left(s_{k-1}-\sum_{\substack{1\leq i_1<\cdots< i_{k-1}\leq n\\i=i_j\text{ for some }1\leq j\leq k-1}}\lambda_{i_1}\cdots\lambda_{i_{k-1}}\right)\otimes t_i\\
				                                                                 &=\sum_{i=1}^n s_{k-1}\lambda_i\otimes t_i-\lambda_i^2\sum_{\substack{1\leq i_1<\cdots< i_{k-1}\leq n\\i\neq i_j\text{ for all }1\leq j\leq k-2}}\lambda_{i_1}\cdots\lambda_{i_{k-2}}\otimes t_i\\
				                                                                 &\vdots\\
				                                                                 &=\sum_{i=1}^n\sum_{j=1}^k(-1)^{j+1}s_{k-j}\lambda_i^j\otimes t_i.
			\end{align*}
		Next, let $v$ and $w\in(R(T)\otimes K^{-1}(T))^{\bigoplus k}$ be defined by 
		\[v_j=\sum_{i=1}^n\lambda_i^j\otimes t_i, \ \ w_j=p_G^*\delta_G(\bigwedge\nolimits^j\sigma_n),\text{ for }1\leq j\leq k.\]
		Define the matrix $M\in M_{k\times k}(R(T))$ to be 
		\[M_{ij}=\begin{cases}0\ \ &\text{if }i<j\\ (-1)^{j+1}\ \ &\text{if }i=j\\(-1)^{j+1}s_{i-j}\ \ &\text{if }i>j\end{cases}.\]
		Equation (\ref{esym}) can be rewritten succinctly as 
			\[w=Mv.\]
			Note that 
			\[(M^{-1})_{ij}=\begin{cases}0\ \ &\text{if }i<j\\(-1)^{j+1}\ \ &\text{if }i=j\\ (-1)^{j+1}h_{i-j}\ \ &\text{if }i>j\end{cases},\]
			where $h_i$ satisfies
			\[h_{i+1}=s_1h_i-s_2h_{i-1}+\cdots+(-1)^is_{i+1},\]
			which is exactly the recursive definition of complete homogeneous symmetric polynomials in terms of elementary symmetric polynomials. Equation (\ref{hsym}) then follows.
\end{proof}

\begin{proposition}\label{adams_op_comp} In $R(T)\otimes K^*(T)$, we have
	\begin{align*}&\sum_{1\leq i_1<\cdots<i_k\leq n}\lambda_{i_1}^l\cdots\lambda_{i_k}^l\otimes(t_{i_1}+\cdots+t_{i_k})\\=&(-1)^{k}\sum_{p=1}^n(-1)^{p}\sum_{\substack{0\leq k_r\leq l-1\\k_1+\cdots+k_n=lk-p}}\lambda_1^{k_1}\cdots\lambda_n^{k_n}p_G^*\delta_G(\bigwedge\nolimits^{p}\sigma_n).\end{align*}
\end{proposition}
\begin{proof} The LHS of the equation can be rewritten as 
	\begin{align}
		  &\sum_{1\leq i_1<\cdots<i_k\leq n}\lambda_{i_1}^l\cdots\lambda_{i_k}^l\otimes(t_{i_1}+\cdots+t_{i_k})\nonumber\\
		=&\sum_{i=1}^n\sum_{j=1}^k(-1)^{j+1}\left(\sum_{1\leq i_1<\cdots<i_{k-j}\leq n}\lambda_{i_1}^l\cdots\lambda_{i_{k-j}}^l\right)\lambda_i^{lj}\otimes t_i\nonumber\\
		&(\text{By Proposition }\ref{sympoly}, \text{ Equation }(\ref{esym}))\nonumber\\
		=&\sum_{j=1}^k\sum_{i=1}^n(-1)^{j+1}\left(\sum_{1\leq i_1<\cdots<i_{k-j}\leq n}\lambda_{i_1}^l\cdots\lambda_{i_{k-j}}^l\right)\lambda_i^{lj}\otimes t_i\nonumber\\
		=&\sum_{j=1}^k(-1)^{j+1}\left(\sum_{1\leq i_1<\cdots<i_{k-j}\leq n}\lambda_{i_1}^l\cdots\lambda_{i_{k-j}}^l\right)\sum_{i=1}^{lj}(-1)^{i+1}h_{lj-i}p_G^*\delta(\bigwedge\nolimits^i\sigma_n)\nonumber\\
		&(\text{By Proposition }\ref{sympoly}, \text{ Equation }  (\ref{hsym}))\nonumber\\
		=&\sum_{m=1}^k\sum_{j=m}^k\sum_{p=1}^l(-1)^{j+1}\left(\sum_{1\leq i_1<\cdots<i_{k-j}\leq n}\lambda_{i_1}^l\cdots\lambda_{i_{k-j}}^l\right)\cdot\nonumber\\
		&(-1)^{lm-l+p+1}h_{lj-lm+l-p}p_G^*\delta(\bigwedge\nolimits^{lm-l+p}\sigma_n).\label{eq1}
	\end{align}
	We shall first compute
	\[\sum_{j=m}^k(-1)^{j+1}\left(\sum_{1\leq i_1<\cdots<i_{k-j}\leq n}\lambda_{i_1}^l\cdots\lambda_{i_{k-j}}^l\right)h_{lj-lm+l-p}.\]
	Consider the identity
	\begin{eqnarray}\label{polyid}(1-\lambda_1^l)\cdots(1-\lambda_n^l)\left(\frac{1}{(1-\lambda_1)\cdots(1-\lambda_n)}\right)=\prod_{i=1}^n(1+\lambda_i+\cdots+\lambda_i^{l-1}).\end{eqnarray}
	Comparing the degree $lk-lm+l-p$ terms of both sides, we have
	\begin{align*}
	&\sum_{j=m}^k(-1)^{k-j}\left(\sum_{1\leq i_1<\cdots<i_{k-j}\leq n}\lambda_{i_1}^l\cdots\lambda_{i_{k-j}}^l\right)h_{lj-lm+l-p}\\=&\sum_{\substack{0\leq k_r\leq l-1\\k_1+\cdots+k_n=lk-lm+l-p}}\lambda_1^{k_1}\cdots\lambda_n^{k_n}, \text{ and}\\
	&\sum_{j=m}^k(-1)^{j+1}\left(\sum_{1\leq i_1<\cdots<i_{k-j}\leq n}\lambda_{i_1}^l\cdots\lambda_{i_{k-j}}^l\right)h_{lj-lm+l-p}\\=&(-1)^{k+1}\sum_{\substack{0\leq k_r\leq l-1\\k_1+\cdots+k_n=lk-lm+l-p}}\lambda_1^{k_1}\cdots\lambda_n^{k_n}.
	\end{align*}
	Substituting into Equation (\ref{eq1}), we have that the LHS of the equation in the Proposition is
	\[\sum_{m=1}^k\sum_{p=1}^l(-1)^{lm-l+p+k}\sum_{\substack{0\leq k_r\leq l-1\\k_1+\cdots+k_n=lk-lm+l-p}}\lambda_1^{k_1}\cdots\lambda_n^{k_n}p_G^*\delta_G(\bigwedge\nolimits^{lm-l+p}\sigma_n).\]
	Replacing $lm-l+p$ by the dummy variable $p$ and noting that $\bigwedge\nolimits^p\sigma_n=0$ for $p>n$, we obtain the RHS. 
\end{proof}

\begin{proof}[Proof of Theorem \ref{mainthm} (\ref{unitary})] Proposition \ref{adams_op_comp} together with the injectivity of $p_G^*$ (cf. Proposition \ref{K_G} (\ref{inj})) yields
\begin{eqnarray}\label{eqadamop}\psi^l(\delta_G(\bigwedge\nolimits^k\sigma_n))=(-1)^kl\sum_{p=1}^n(-1)^p\sum_{\substack{0\leq k_r\leq l-1\\k_1+\cdots+k_n=lk-p}}\lambda_1^{k_1}\cdots\lambda_n^{k_n}\delta_G(\bigwedge\nolimits^p\sigma_n).\end{eqnarray}
Applying the forgetful map, which on the RHS of (\ref{eqadamop}) amounts to specializing at $\lambda_1=\cdots=\lambda_n=1$, yields Equation (\ref{unitaryadamsop}). Comparing the degree $lk-p$ terms of both sides of Equation (\ref{polyid}), we have 
\begin{eqnarray}\label{adamscoeff}\sum_{q=0}^{k-1}(-1)^q\left(\sum_{1\leq i_1<\cdots< i_q\leq n}\lambda_{i_1}^l\cdots\lambda_{i_q}^l\right)h_{l(k-q)-p}=\sum_{\substack{0\leq k_r\leq l-1\\k_1+\cdots+k_n=lk-p}}\lambda_1^{k_1}\cdots\lambda_n^{k_n}.\end{eqnarray}
Specializing at $\lambda_1=\cdots=\lambda_n=1$, we get Equation (\ref{adamsopcoeff}). 
%\[\mu(n, l, k, p)=\sum_{q=0}^{k-1}(-1)^q\binom{n}{q}\binom{n+l(k-q)-p-1}{n-1}.\]
%This, together with the application of the forgetful map to both sides of Equation (\ref{eqadamop}), yields the equation we set to prove. 
When $l=2$, $\mu(n, l, k, p)$ can be easily seen to be $\displaystyle\binom{n}{2k-p}$, and this shows the formula in the special case $\psi^2$.
\end{proof}
\begin{remark}
	\begin{enumerate}
		\item In \cite{N}, Naylor noted that each coefficient in the formula for general $\psi^l$ on $SU(n)$ is a polynomial in $l$. He then used the formula for $\psi^2$ repeatedly to obtain the one for $\psi^{2^n}$, implemented Lagrange interpolation at $l=2, \cdots, 2^n$ and wrote down a formula for $\psi^l$. Our formula (Equation (\ref{unitaryadamsop})) agrees with his formula when $l=2$, and in general is more explicit. 
		\item One can see that our formula (\ref{unitaryadamsop}) does satisfy $\psi^m\psi^l=\psi^{ml}$ as follows. On the one hand we have
		\[\psi^m\psi^l(\delta(\bigwedge\nolimits^k\sigma_n))=(-1)^{k}ml\sum_{q=1}^n\sum_{p=1}^n(-1)^q\mu(n, l, k, p)\mu(n, m, p, q)\delta(\bigwedge\nolimits^q\sigma_n).\]
		On the other hand, we have $\displaystyle \sum_{p=1}^n\mu(n, l, k, p)\mu(n, m, p, q)=\mu(n, ml, k, q)$ because of the bijection
		\begin{align*}
			\bigcup_{p=1}^n&\{(k_1, \cdots, k_n)\in\mathbb{Z}^{\oplus n}|\ 0\leq k_r\leq l-1,\ k_1+\cdots+k_n=lk-p\}\times\\
						 &\{(k_1', \cdots, k_n')\in\mathbb{Z}^{\oplus n}|\ 0\leq k_r\leq m-1,\ k_1'+\cdots+k_n'=mp-q\}\\
						 &\longrightarrow \{(k_1'', \cdots, k_n'')\in\mathbb{Z}^{\oplus n}|\ 0\leq k_r''\leq ml-1,\ k_1''+\cdots+k_n''=mlk-q\}
		\end{align*}
		given by $((k_1, \cdots, k_n), (k_1', \cdots, k_n'))\mapsto (mk_1+k_1', \cdots, mk_n+k_n')$.
		\item The expression $\mu(n, l, k, p)$ in our formula (\ref{unitaryadamsop}) does satisfy the equation $\mu(n, l, k, p)=\mu(n, l, n-k, n-p)$ (cf. \cite[Theorem 2.3]{N}). This can be deduced from the bijection
			\begin{align*}
				&\{(k_1, \cdots, k_n)\in\mathbb{Z}^{\oplus n}|\ 0\leq k_r\leq l-1,\ k_1+\cdots+k_n=lk-p\}\\
				\longrightarrow &\{(k_1', \cdots, k_n')\in\mathbb{Z}^{\oplus n}|\ 0\leq k_r'\leq l-1,\ k_1'+\cdots+k_n'=l(n-k)-(n-p)\}
			\end{align*}
		given by $(k_1, \cdots, k_n)\mapsto (l-1-k_1, \cdots, l-1-k_n)$. 
		\item Though we are unable to see directly, without using Theorem \ref{mainthm} (\ref{unitary}), that the formula for Adams operations on $SU(n)$ in \cite{N} agrees with ours in general, the formulas for Adams operations on $SU(n)$, $n=3, 4$ and $5$ given immediately after \cite[Theorem 2.3]{N} are easily seen to be consistent with ours.
	\end{enumerate}
\end{remark}
%\begin{proposition}\label{chernadams}
%	Let $X$ be a finite CW-complex and $\alpha\in K^{-1}(X)$. Then we have 
%	\[\text{ch}(\psi^l(\alpha))=\sum_il^i\text{ch}_{2i-1}(\alpha)\]
%	where $\text{ch}: K^*(X)\to H^*(X, \mathbb{Q})$ is the Chern character and $\text{ch}_k(\alpha)$ is the degree $k$ term of $\text{ch}(\alpha)$. 
%\end{proposition}	
%\begin{proof}
%	Let $\beta\in K^0(SX)$ be the image of $\alpha$ under the suspension isomorphism. As $\text{ch}$ respects suspension isomorphism, it suffices to show that
%	\[\text{ch}(\psi^l(\beta))=\sum_i l^i\text{ch}_{2i}(\beta)\]
%	We may assume that $\beta=\sum_{j=1}^mL_j$ where $L_j$ are $K$-theory classes of line bundles by splitting principle. It follows that 
%	\begin{align*}
%		\text{ch}(\psi^l(\beta))&=\text{ch}\left(\psi^l\left(\sum_{j=1}^m L_j\right)\right)\\
%		&=\sum_{j=1}^m\text{ch}(L_j^{\otimes l})\\
%		&=\sum_{j=1}^m\sum_i\frac{l^i c_1(L_j)^i}{i!}\\
%		&=\sum_i l^i\text{ch}_{2i}(\beta)
%	\end{align*}
%\end{proof}

\begin{definition}
Let $PK^{-1}(G)$ be the vector space over $\mathbb{Q}$ spanned by the primitive elements of $K^*(G)$. 
\end{definition}
It is known that, for a general compact connected Lie group $G$ of rank $n$, its exponents are $m_i$, $i=1, \cdots, n$ if and only if its rational cohomology ring $H^*(G, \mathbb{Q})$ is an exterior algebra generated by primitive elements of degrees $2m_i+1$, $i=1, \cdots, n$ (cf. \cite{R}). In \cite[Section 2]{N} it was shown, by means of the Chern character isomorphism, that the Adams operation $\psi^l$ on $PK^{-1}(SU(n))$ has eigenvalues $l^i$, $i=2, \cdots, n$, each with multiplicity 1, and that the possible values of $i$ are exactly the exponents of $SU(n)$ plus 1. One can get the following more general result on the eigenvalues of $\psi^l$ on $PK^{-1}(G)$ by adapting the arguments in \cite{N}.
\begin{proposition}\label{polygrowth}
	Let $G$ be a compact Lie group of rank $n$ with torsion-free fundamental group. If the exterior algebra $H^*(G, \mathbb{Q})$ is generated by primitive elements of degrees $2m_i+1$, $1\leq i\leq n$, then the eigenvalues of $\psi^l\otimes\text{Id}_\mathbb{Q}: PK^{-1}(G)\to PK^{-1}(G)$ are $l^{m_1+1}, l^{m_2+1}, \cdots, l^{m_n+1}$, and the entries of the matrix representation of $\psi^l$ with respect to any given basis of $PK^{-1}(G)$ are polynomials in $l$ of degrees at most $\max\{m_1+1, \cdots, m_n+1\}$. 
\end{proposition}
%\begin{proof}
%	Let $v\in PK^{-1}(G)$ be an eigenvector of $\psi^l\otimes\text{Id}_\mathbb{Q}$ corresponding to eigenvalue $r$. Applying $\text{ch}$ to both sides of $\psi^l(v)=rv$ and Proposition \ref{chernadams}, we have
%	\[\sum_{i=1}^nl^{m_i+1}\text{ch}_{2m_i+1}(v)=r\sum_{i=1}^n\text{ch}_{2m_i+1}(v)\]
%	So there exists $j$ between 1 and $n$ such that $\text{ch}_{2m_i+1}(v)=0$ for $i\neq j$ and $r=l^{m_j+1}$.
%\end{proof}
%\begin{corollary}\label{polygrowth}
%\end{corollary}
For $G=U(n)$, $m_i=i-1$ for $i=1, \cdots, n$. Then $\psi^l\otimes\text{Id}_\mathbb{Q}$ has distinct eigenvalues and hence it is diagonalizable. Moreover, Adams operations commute with each other. Consequently there exists a basis of common eigenvectors for $\psi^l\otimes\text{Id}_\mathbb{Q}$ (cf. \cite[Section 2]{N}). To find out those vectors it suffices to work with $\psi^2$, which has a simpler form.
\begin{proof}[Proof of Theorem \ref{mainthm} (\ref{eigenvector})]
We shall first prove the recurrence relation for the polynomial sequence. It is easy to see that $p_0(y)=1$. Note that 
\begin{eqnarray}\label{series}
	\sum_{j=0}^\infty p_j(y)t^{2j}=\left(\frac{t}{\sinh t}\right)^y=\text{exp}(yg(t)), 
\end{eqnarray}
where $\displaystyle g(t)=\ln\left(\frac{t}{\sinh t}\right)$. Moreover, 
\[g'(t)=\frac{1}{t}-\coth t=-\sum_{k=1}^\infty\frac{2^{2k}B_{2k}}{(2k)!}t^{2k-1}.\]
Differentiating both sides of Equation (\ref{series}) yields
\[\sum_{j=0}^\infty 2jp_j(y)t^{2j-1}=\text{exp}(yg(t))yg'(t)=\left(\sum_{n=0}^\infty p_n(y)t^{2n}\right)\left(-y\sum_{k=1}^\infty\frac{2^{2k}B_{2k}}{(2k)!}t^{2k-1}\right).\]
Comparing the coefficients of $t^{2j-1}$ of both sides gives the desired recurrence relation. 

To show that the expression (\ref{eigvec}) in Theorem \ref{mainthm} (\ref{eigenvector}) is an eigenvector of Adams operations on $K^*(U(n))\otimes\mathbb{Q}$, it suffices to prove that 
\begin{align}\label{eigwts}
	&\psi^2\left(\sum_{i=1}^n (-1)^{i-1}\left(\sum_{j=0}^{\lfloor\frac{k}{2}\rfloor}\frac{p_j(n)}{(k-2j)!}(n-2i)^{k-2j}\right)\delta(\bigwedge\nolimits^i\sigma_n)\right)\nonumber\\
	=& 2^{n-k} \sum_{i=1}^n (-1)^{i-1}\left(\sum_{j=0}^{\lfloor\frac{k}{2}\rfloor}\frac{p_j(n)}{(k-2j)!}(n-2i)^{k-2j}\right)\delta(\bigwedge\nolimits^i\sigma_n).
\end{align}
From the LHS of Equation (\ref{eigwts}), we have
\begin{align}\label{simplhs}
	&\psi^2\left(\sum_{i=1}^n (-1)^{i-1}\left(\sum_{j=0}^{\lfloor\frac{k}{2}\rfloor}\frac{p_j(n)}{(k-2j)!}(n-2i)^{k-2j}\right)\delta(\bigwedge\nolimits^i\sigma_n)\right)\\
	=&\sum_{i=1}^n (-1)^{i-1}\left(\sum_{j=0}^{\lfloor\frac{k}{2}\rfloor}\frac{p_j(n)}{(k-2j)!}(n-2i)^{k-2j}\right)\cdot (-1)^i\cdot 2\sum_{p=1}^{2i}(-1)^p\binom{n}{2i-p}\delta(\bigwedge\nolimits^p\sigma_n)\nonumber\\
	&(\text{By Theorem \ref{mainthm} (\ref{unitary})})\nonumber\\
	=&\sum_{i=1}^n\sum_{j=0}^{\lfloor\frac{k}{2}\rfloor}\sum_{p=1}^{2i}(-1)^{p-1}\frac{2p_j(n)}{(k-2j)!}(n-2i)^{k-2j}\binom{n}{2i-p}\delta(\bigwedge\nolimits^p\sigma_n)\nonumber\\
	=&\sum_{p=1}^n\sum_{i=\lceil\frac{p}{2}\rceil}^n\sum_{j=0}^{\lfloor\frac{k}{2}\rfloor}(-1)^{p-1}\frac{2p_j(n)}{(k-2j)!}(n-2i)^{k-2j}\binom{n}{2i-p}\delta(\bigwedge\nolimits^p\sigma_n).\nonumber
\end{align}
 Suppose $k$ is an odd number. The sum $\displaystyle \sum_{j=0}^{\lfloor\frac{k}{2}\rfloor}\frac{p_j(n)}{(k-2j)!}(n-2i)^{k-2j}$ then is the coefficient of $t^k$ in the Taylor series of $\displaystyle\left(\frac{t}{\sinh t}\right)^n\sinh ((n-2i)t)$.
 Using $(f(t))_k$ to denote the coefficient of $t^k$ in the Taylor series of $f(t)$, we can simplify the coefficient of $\delta(\bigwedge\nolimits^p\sigma_n)$ in Equation (\ref{simplhs}) as follows:
 \begin{align*}
 	&\sum_{i=\lceil\frac{p}{2}\rceil}^n\sum_{j=0}^{\lfloor\frac{k}{2}\rfloor}(-1)^{p-1}\frac{2p_j(n)}{(k-2j)!}(n-2i)^{k-2j}\binom{n}{2i-p}\\
	=&\sum_{i=\lceil\frac{p}{2}\rceil}^n (-1)^{p-1}\cdot 2\left(\left(\frac{t}{\sinh t}\right)^n\sinh ((n-2i)t)\binom{n}{2i-p}\right)_k\\
	=&(-1)^{p-1}\sum_{i=\lceil \frac{p}{2}\rceil}^n \left(\left(\frac{t}{\sinh t}\right)^n\cdot\binom{n}{2i-p}(e^{(n-2i)t}-e^{-(n-2i)t})\right)_k\\
	=&(-1)^{p-1}\cdot\frac{1}{2}\Bigg(\left(\frac{t}{\sinh t}\right)^n\cdot(e^{-tp}((1+e^t)^n+(-1)^p(1-e^t)^n)-\\
	 &e^{tp}((1+e^{-t})^n+(-1)^p(1-e^{-t})^n))\Bigg)_k\\
	=&(-1)^{p-1}\cdot\frac{1}{2}\Bigg(\left(\frac{t}{\sinh t}\right)^n\cdot\Bigg(e^{t\left(\frac{n}{2}-p\right)}\cdot 2^n\left(\cosh ^n\frac{t}{2}+(-1)^{p+n}\sinh^n\frac{t}{2}\right)-\\
	&e^{t\left(p-\frac{n}{2}\right)}\cdot 2^n\left(\cosh^n\frac{t}{2}+(-1)^p\sinh^n\frac{t}{2}\right)\Bigg)\Bigg)_k\\
	=&(-1)^{p-1}\cdot 2^{n-1}\left(\frac{t^ne^{t\left(\frac{n}{2}-p\right)}}{2^n\sinh^n\frac{t}{2}}+\frac{(-1)^{p+n}t^ne^{t\left(\frac{n}{2}-p\right)}}{2^n\cosh^n\frac{t}{2}}-\frac{t^ne^{t\left(p-\frac{n}{2}\right)}}{2^n\sinh^n\frac{t}{2}}-\frac{(-1)^pt^ne^{t\left(p-\frac{n}{2}\right)}}{2^n\cosh^n\frac{t}{2}}\right)_k\\
	=&(-1)^{p-1}\cdot 2^n\Bigg(\left(\frac{\frac{t}{2}}{\sinh \frac{t}{2}}\right)^n\sinh \left((n-2p)\frac{t}{2}\right)+\\
	&(-1)^p\left(\frac{\frac{t}{2}}{\cosh\frac{t}{2}}\right)^n\cdot\frac{(-1)^ne^{\frac{t}{2}(n-2p)}-e^{\frac{t}{2}(2p-n)}}{2}\Bigg)_k\\
	=&(-1)^{p-1}\cdot 2^n\left(\left(\frac{\frac{t}{2}}{\sinh \frac{t}{2}}\right)^n\sinh \left((n-2p)\frac{t}{2}\right)+O(t^n)\right)_k\\
	=&(-1)^{p-1}\cdot 2^n\left(\left(\frac{\frac{t}{2}}{\sinh \frac{t}{2}}\right)^n\sinh \left((n-2p)\frac{t}{2}\right)\right)_k\ \ (\text{Note that }k<n)\\
	=&(-1)^{p-1}\cdot 2^{n-k}\left(\left(\frac{t}{\sinh t}\right)^n\sinh ((n-2p)t)\right)_k\\
	=&(-1)^{p-1}\cdot 2^{n-k}\sum_{j=0}^{\lfloor\frac{k}{2}\rfloor}\frac{p_j(n)}{(k-2j)!}(n-2p)^{k-2j}.
 \end{align*}
This finishes the proof for the case of $k$ being odd. If $k$ is even, the proof proceeds in a similar fashion. The only difference one needs to bear in mind when proving this case is that $\displaystyle \sum_{j=0}^{\lfloor\frac{k}{2}\rfloor}\frac{p_j(n)}{(k-2j)!}(n-2i)^{k-2j}=\left(\left(\frac{t}{\sinh t}\right)^n\cosh ((n-2i)t)\right)_k$. We leave the details to the reader. 
\end{proof}
\begin{remark}
	\begin{enumerate}
		\item The function $\displaystyle\left(\frac{t}{\sinh t}\right)^y$, which defines the polynomial sequence $\{p_j(y)\}_{j=0}^\infty$, apparently bears tantalizing resemblance to the function $\displaystyle\frac{\frac{\sqrt{t}}{2}}{\sinh \frac{\sqrt{t}}{2}}$, which defines $\widehat{A}$-genera. This suggests a possible alternative proof of the eigenvector formula by means of index theory, but we have yet to find such a proof.
		\item In \cite{N} Naylor outlined a recursive algorithm for computing the eigenvectors of Adams operations on $U(n)$ as follows. Assuming the knowledge of an eigenvector $v_k\in PK^{-1}(U(n))$ corresponding to the eigenvalue $l^{n-k}$, one can find the eigenvector $v'_{k+1}\in PK^{-1}(U(n+1))$ which corresponds to the same eigenvalue and restricts to $v_k$ through the pullback map induced by the natural inclusion $U(n)\to U(n+1)$. This can be done using the branching law of $U(n)$. Replacing $n+1$ appearing in the formula of $v'_{k+1}$ so obtained with $n$ gives the formula for $v_{k+1}\in PK^{-1}(U(n))$. Though he claimed that the algorithm is efficient, no explicit formulas were given. 
	\end{enumerate}
\end{remark}

\section{$G=Sp(n)$, $Spin(n)$ or $G_2$}
Let $\rho$ be an $n$-dimensional representation of a compact Lie group $G$. The functoriality of Adams operations and Theorem \ref{mainthm} (\ref{unitary}) enable us to compute $\psi^l(\delta(\rho))$. More precisely, the induced map $\rho^*: K^*(U(n))\to K^*(G)$ pulls the formula for Adams operations on $K^*(U(n))$ back to the one for $\psi^l(\delta(\rho))$. 

	For $G=Sp(n)$, the isomorphism classes of representations 
	\[\{\sigma_{2n}, \bigwedge\nolimits^2\sigma_{2n}, \cdots, \bigwedge\nolimits^n\sigma_{2n}\}\subset R(Sp(n))\]
	 form a $\mathbb{Z}$-module basis of the indecomposables $I/I^2$, which thus differs from the basis consisting of the fundamental representations by a unimodular linear transformation. By Remark \ref{primindecomp} and the main result of \cite{Ho} (see Theorem \ref{K_G} (1)), $K^*(Sp(n))\cong\bigwedge\nolimits^*_\mathbb{Z}(\delta(\sigma_{2n}), \delta(\bigwedge\nolimits^2\sigma_{2n}), \cdots, \delta(\bigwedge\nolimits^n\sigma_{2n}))$. In Theorem \ref{symplectic}, we choose to express the formula for Adams operations on $K^*(Sp(n))$ in terms of the generators $\delta(\sigma_{2n})$, $\delta(\bigwedge\nolimits^2(\sigma_{2n})), \cdots, \delta(\bigwedge\nolimits^n\sigma_{2n})$ instead of those involving fundamental representations for aesthetic reasons. 
\begin{proof}[Proof of Theorem \ref{symplectic}]	
	 Let $\rho$ be the standard representation $\sigma_n$ of $Sp(n)$. By applying the pullback map $\rho^*: K^*(U(n))\to K^*(Sp(n))$ to Equation (\ref{unitaryadamsop}), and noting that $\bigwedge\nolimits^p\sigma_{2n}$ and $\bigwedge\nolimits^{2n-p}\sigma_{2n}$ are isomorphic representations of $Sp(n)$, we have the formula in Theorem \ref{symplectic}.
\end{proof}
%\begin{align*}
%	\psi^l(\delta(\bigwedge\nolimits^k\sigma_{2n}))&=(-1)^kl \left(\sum_{p=1}^{n-1}\sum_{q=0}^{k-1}(-1)^{p+q}\binom{2n}{q}\left(\binom{2n+l(k-q)-p-1}{2n-1}+\right.\right.\\
%									      &\left.\left.\binom{l(k-q)+p-1}{2n-1}\right)\delta(\bigwedge\nolimits^p\sigma_{2n})\right.\\
%									       &\left.+\sum_{q=0}^{k-1}(-1)^{n+q}\binom{2n}{q}\binom{n+l(k-q)-1}{2n-1}\delta(\bigwedge\nolimits^n\sigma_{2n})\right)
%\end{align*}
The proof for the spin case is similar, but we need some more representation theory to deal with the (half) spin representations. This complication explains the length of the formulas in Theorem \ref{spin}. 
\begin{proof}[Proof of Theorem \ref{spin}]
For $G=Spin(2n+1)$, its $K$-theory is given by 
\[K^*(Spin(2n+1))\cong\bigwedge\nolimits^*_\mathbb{Z}(\delta(\sigma_{2n+1}), \cdots, \delta(\bigwedge\nolimits^{n-1}\sigma_{2n+1}), \delta(S)),\]
where $S$ is the spin representation, whose dimension is $2^n$. The tensor square of $S$ decomposes as (cf. \cite[Exercise 19.16]{FH})
\[S^{\otimes 2}=\bigoplus_{p=0}^n\bigwedge\nolimits^p\sigma_{2n+1}.\]
It follows that
\begin{align}
	\delta(\bigwedge\nolimits^n\sigma_{2n+1})&=\delta(S^{\otimes 2})-\sum_{p=1}^{n-1}\delta(\bigwedge\nolimits^p\sigma_{2n+1})\nonumber\\
	                                                                    &=2^{n+1}\delta(S)-\sum_{p=1}^{n-1}\delta(\bigwedge\nolimits^p\sigma_{2n+1}).\label{spinrep}
\end{align}
Theorem \ref{mainthm} (\ref{unitary}), the functoriality of Adams operations, the fact that $\bigwedge\nolimits^p\sigma_{2n+1}\cong\bigwedge\nolimits^{2n+1-p}\sigma_{2n+1}$ and Equation (\ref{spinrep}) yield the formulas in Theorem \ref{spin} (\ref{spinodd}). 
%\begin{align*}
%	\psi^l(\delta(\bigwedge\nolimits^k\sigma_{2n+1}))&=(-1)^k l\left(\sum_{p=1}^{n-1}\sum_{q=0}^{k-1}(-1)^q\binom{2n+1}{q}\left((-1)^p\left(\binom{2n+l(k-q)-p}{2n}-\binom{l(k-q)+p-1}{2n}\right)\right.\right.\\
%	                                                                   &-(-1)^n\left.\left(\binom{n+l(k-q)}{2n}-\binom{n+l(k-q)-1}{2n}\right)\right)\delta(\bigwedge\nolimits^p\sigma_{2n+1})\\
%	                                                                   &\left.+\sum_{q=0}^{k-1}(-1)^{n+q}\cdot 2^{n+1}\binom{2n+1}{q}\left(\binom{n+l(k-q)}{2n}-\binom{n+l(k-q)-1}{2n}\right)\delta(S)\right)\\
%	                                        \psi^l(\delta(S))&=\frac{1}{2^{n+1}}\left(\sum_{k=1}^n (-1)^k l\left(\sum_{p=1}^{n-1}\sum_{q=0}^{k-1}(-1)^q\binom{2n+1}{q}\left((-1)^p\binom{2n+l(k-q)-p}{2n}-\right.\right.\right.\\
%	                                                                &\left.(-1)^n\binom{n+l(k-q)}{2n}\right)\delta(\bigwedge\nolimits^p\sigma_{2n+1})+\\
%	                                                                &\left.\left.\sum_{q=0}^{k-1}(-1)^{n+q}\cdot 2^{n+1}\binom{2n+1}{q}\binom{n+l(k-q)}{2n}\delta(S)\right)\right)
%\end{align*}

Similarly, for $G=Spin(2n)$, its $K$-theory is given by 
\[K^*(Spin(2n))\cong\bigwedge\nolimits^*(\delta(\sigma_{2n}), \cdots, \delta(\bigwedge\nolimits^{n-2}\sigma_{2n}), \delta(S_+), \delta(S_-)),\]
where $S_+$ and $S_-$ are positive and negative half spin representations respectively, whose dimensions are both $2^{n-1}$. Note that (cf. \cite[Exercises 19.6 and 19.7]{FH})
\begin{align}
	S_+^{\otimes 2}\oplus S_-^{\otimes 2}&=\bigoplus_{p\geq 1}\left(\bigwedge\nolimits^{n-2p}\sigma_{2n}\right)^{\oplus 2}\oplus \bigwedge\nolimits^n\sigma_{2n}, \text{ and}\label{sumofsq}\\
	S_+\otimes S_-&=\bigoplus_{p\geq 0}\bigwedge\nolimits^{n-2p-1}\sigma_{2n}.\label{tensorprodpm}
\end{align}
Applying $\delta$ to the above equations and using its derivation property lead to 
\begin{align}
	2^n(\delta(S_+)+\delta(S_-))&=2\sum_{p\geq 1}\delta(\bigwedge\nolimits^{n-2p}\sigma_{2n})+\delta(\bigwedge\nolimits^n\sigma_{2n}),\nonumber\\
	\delta(\bigwedge\nolimits^n\sigma_{2n})&=2^n(\delta(S_+)+\delta(S_-))-2\sum_{p\geq 1}\delta(\bigwedge\nolimits^{n-2p}\sigma_{2n}),\label{dsumofsq}\\
	2^{n-1}(\delta(S_+)+\delta(S_-))&=\sum_{p\geq 0}\delta(\bigwedge\nolimits^{n-2p-1}\sigma_{2n})\label{dtensorprodpma}, \\
	\delta(\bigwedge\nolimits^{n-1}\sigma_{2n})&=2^{n-1}(\delta(S_+)+\delta(S_-))-\sum_{p\geq 1}\delta(\bigwedge\nolimits^{n-2p-1}\sigma_{2n}).\label{dtensorprodpm}
\end{align}
Again Theorem \ref{mainthm} (\ref{unitary}), the functoriality of Adams operations, the fact that $\bigwedge\nolimits^p\sigma_{2n+1}$ $\cong$ $\bigwedge\nolimits^{2n+1-p}\sigma_{2n+1}$ and Equations (\ref{dsumofsq}) and (\ref{dtensorprodpm}) yield the formula for Adams operations on $\delta(\bigwedge\nolimits^k\sigma_{2n})$ in Theorem \ref{spin} (\ref{spineven}).
%\begin{align*}
%				\psi^l(\delta(\bigwedge\nolimits^k\sigma_{2n}))&=(-1)^kl\left(\sum_{p\geq 1}\sum_{q=0}^{k-1}(-1)^{n+q}\binom{2n}{q}\left(\binom{n+l(k-q)+2p-1}{2n-1}\right.\right.\\
%												       &\left.+\binom{n+l(k-q)-2p-1}{2n-1}-2\binom{n+l(k-q)-1}{2n-1}\right)\delta(\bigwedge\nolimits^{n-2p}\sigma_{2n})\\
%												       &+\sum_{p\geq 1}\sum_{q=0}^{k-1}(-1)^{n+1+q}\binom{2n}{q}\left(\binom{n+l(k-q)+2p}{2n-1}+\binom{n+l(k-q)-2p-2}{2n-1}\right.\\
%												       &\left.-\binom{n+l(k-q)}{2n-1}-\binom{n+l(k-q)-2}{2n-1}\right)\delta(\bigwedge\nolimits^{n-1-2p}\sigma_{2n})\\
%												       &+\sum_{q=0}^{k-1}(-1)^{n+1+q}\cdot 2^{n-1}\binom{2n}{q}\left(\binom{n+l(k-q)}{2n-1}+\binom{n+l(k-q)-2}{2n-1}\right.\\
%												       &\left.\left.-2\binom{n+l(k-q)-1}{2n-1}\right)(\delta(S_+)+\delta(S_-))\right)
%\end{align*}

The computation of $\psi^l(\delta(S_\pm))$ requires a bit more work. By Theorem \ref{mainthm} (\ref{unitary}) and functoriality of Adams operations, we have
\[\psi^2(\delta(S_\pm))=2^n\delta(S_\pm)-2\delta(\bigwedge\nolimits^2 S_\pm).\]
The two exterior squares $\bigwedge\nolimits^2 S_+$ and $\bigwedge\nolimits^2 S_-$ are both isomorphic to $\bigoplus_{i\geq 0}\bigwedge\nolimits^{n-4i-2}\sigma_{2n}$ (cf. \cite[p. 63]{Bot}). Hence
\[\psi^2(\delta(S_+)-\delta(S_-))=2^n(\delta(S_+)-\delta(S_-)).\]
Applying $\psi^2$ repeatedly gives
\[\psi^{2^k}(\delta(S_+)-\delta(S_-))=2^{kn}(\delta(S_+)-\delta(S_-)).\]
Using Proposition \ref{polygrowth} and Lagrange interpolation at $l=2^k$ for various $k$, we have in general
\begin{eqnarray}\label{adamsopdiff}
	\psi^l(\delta(S_+)-\delta(S_-))=l^n(\delta(S_+)-\delta(S_-)).
\end{eqnarray}
Besides, by Equation (\ref{dtensorprodpma}), Theorem \ref{mainthm} (\ref{unitary}), functoriality of Adams operations and the fact that $\bigwedge\nolimits^p\sigma_{2n}\cong\bigwedge\nolimits^{2n-p}\sigma_{2n}$, we have
\begin{align}
		&\psi^l(\delta(S_+)+\delta(S_-))\nonumber\\
	=&\left(\frac{1}{2}\right)^{n-1}l\sum_{k\geq 0}\sum_{p\geq 1}((\alpha(2n, l, n-2k-1, n-2p-1)\nonumber\\
					       &-\alpha(2n, l, n-2k-1, n-1))\delta(\bigwedge\nolimits^{n-2p-1}\sigma_{2n})\nonumber\\
					       &-(\alpha(2n, l, n-2k-1, n-2p)-2\mu(2n, l, n-2k-1, n))\delta(\bigwedge\nolimits^{n-2p}\sigma_{2n}))\nonumber\\
					       &+l\sum_{k\geq 0}((\alpha(2n, l, n-2k-1, n-1)-2\mu(2n, l, n-2k-1, n))(\delta(S_+)+\delta(S_-)).\label{adamsopsum}
\end{align}
The desired formulas for $\psi^l(\delta(S_\pm))$ in Theorem \ref{spin} (\ref{spineven}) then follow from Equations (\ref{adamsopdiff}) and (\ref{adamsopsum}).
\end{proof}
\begin{remark}
	The formulas in Theorem \ref{spin}, as they appear, involve fractional coefficients. They are in fact integers, as Adams operations act on \emph{integral} $K$-theory. It seems to be an interesting combinatorial and number theoretic fact in its own right, but we have yet to find a direct proof of it.
\end{remark}

Finally, we recover the formulas for Adams operations on $K^*(G_2)$ in \cite{W3}, where they were derived by means of the Chern character isomorphism. Let $\rho_1$ and $\rho_2$ be the fundamental representations of $G_2$ with dimensions 7 and 14 respectively. Note that $\bigwedge\nolimits^2\rho_1=\bigwedge\nolimits^5\rho_1=\rho_1+\rho_2$ and $\bigwedge\nolimits^3\rho_1=\bigwedge\nolimits^4\rho_1=\rho_1^2-\rho_2$. By the derivation property of $\delta$, $\delta(\bigwedge\nolimits^3\rho_1)=\delta(\bigwedge\nolimits^4\rho_1)=14\delta(\rho_1)-\delta(\rho_2)$. Reprising the method of using Theorem \ref{mainthm} and the functoriality of Adams operations, we get
\begin{align*}
	&\psi^l(\delta(\rho_1))=l\sum_{p=1}^6(-1)^{p+1}\mu(7, l, 1, p)\delta(\bigwedge\nolimits^p\rho_1)=\frac{2l^6+13l^2}{15}\delta(\rho_1)+\frac{l^2-l^6}{30}\delta(\rho_2), \\
	                                         %&=l\left(\binom{5+l}{6}\delta(\rho_1)-\binom{4+l}{6}(\delta(\rho_1+\rho_2)+\right.\\
	                                         %&\left.\left(\binom{l+3}{6}-\binom{l+2}{6}\right)(14\delta(\rho_1)-\delta(\rho_2))+\binom{l+1}{6}(\delta(\rho_1+\rho_2))-\binom{l}{6}\delta(\rho_1)\right)\\
	 &\psi^l(\delta(\bigwedge\nolimits^2\rho_1))=l\sum_{p=1}^6(-1)^p\mu(7, l, 2, p)\delta(\bigwedge\nolimits^p\rho_1)=\frac{13l^2-10l^6}{3}\delta(\rho_1)+\frac{5l^6+l^2}{6}\delta(\rho_2),\\	                                                                                
	 %&=l\left(-\left(\binom{5+2l}{6}-7\binom{5+l}{6}\right)\delta(\rho_1)\right.\\
	                                                                                %&+\left(\binom{4+2l}{6}-7\binom{4+l}{6}\right)(\delta(\rho_1+\rho_2))\\
	                                                                                %&-\left(\binom{3+2l}{6}-7\binom{3+l}{6}\right)(14\delta(\rho_1)-\delta(\rho_2))\\
	                                                                                %&+\left(\binom{2+2l}{6}-7\binom{2+l}{6}\right)(14\delta(\rho_1)-\delta(\rho_2))\\
	                                                                                %&-\left(\binom{1+2l}{6}-7\binom{1+l}{6}\right)(\delta(\rho_1+\rho_2))\\
	                                                                                %&+\left.\left(\binom{2l}{6}-7\binom{l}{6}\right)\delta(\rho_1)\right)\\
	                                                                       &\psi^l(\delta(\rho_2))=\psi^l(\delta(\bigwedge\nolimits^2\rho_1))-\psi^l(\delta(\rho_1))=\frac{52l^2(1-l^4)}{15}\delta(\rho_1)+\frac{l^2(13l^4+2)}{15}\delta(\rho_2).
\end{align*}
\begin{remark}
	The formulas in \cite{W3} are expressed with respect to the basis 
	\[\{\delta(\rho_1), \delta(\bigwedge\nolimits^2\rho_1)\}.\]
\end{remark}

\end{document}